\newcommand{\home}{\empty}
\documentclass[12pt]{article}
\usepackage{\home razb}

\newlength{\longeqmarginwidth}
\newlength{\longeqwidth}
\newlength{\longeqskiplength}
\newcommand{\longeq}[1]{
\settowidth{\longeqmarginwidth}{{\Large \{~\}}~(\theequation)}
\setlength{\longeqwidth}{\textwidth}
\addtolength{\longeqwidth}{-\longeqmarginwidth}
\left. \parbox{\longeqwidth}{\begin{eqnarray*} #1
\end{eqnarray*}} \right\}}

\newcommand{\Vzero}[2]{\mathord{\stackrel{\textstyle\kern1pt\circ}
{\smash V\vbox to6pt{}}}\vphantom{V}_{#1}^{#2}}

\newcommand{\function}[2]{:#1 \longrightarrow #2}
\newcommand{\of}[1]{\left( #1 \right)}

\newcommand{\set}[2]{\left\{\hspace{0.2ex} #1 \left|\: #2
\right. \right\}}

\newcommand{\edge}[2]{\langle #1,#2 \rangle}

\newcommand{\df}{\stackrel{\rm def}{=}}

\newcommand{\rn}{\boldsymbol}
\newcommand{\scr}{\mathcal}

\newcounter{operator}

\usepackage{amsfonts}
\usepackage{amsmath}
\usepackage{epic,eepic}
\title{On Turan's (3,4)-problem with forbidden configurations}
\author{Alexander A. Razborov\thanks{University of Chicago, {\tt razborov@cs.uchicago.edu}. Part of this work was done while the author was at Steklov Mathematical Institute, supported by the Russian Foundation
for Basic Research, and at Toyota Technological Institute, Chicago.}}
\begin{document}
\maketitle

\begin{abstract}
We identify three 3-graphs on five vertices each missing in all known extremal configurations for Turan's (3,4)-problem and prove Turan's conjecture for 3-graphs that are additionally known not to contain any induced copies of these 3-graphs. Our argument is based on an (apparently) new technique of ``indirect interpretation'' that allows us to retrieve additional structure from hypothetical counterexamples to Turan's conjecture, but in rather loose and limited sense. We also include two miscellaneous calculations in flag algebras that prove similar results about some other additional forbidden subgraphs.
\end{abstract}

\section{Introduction}

In the classical paper \cite{Man}, Mantel determined the minimal
number of edges a graph $G$ with a given number of vertices must have so that every three vertices
span at least one edge. In the paper
\cite{Tur} (that essentially started off the field of extremal
combinatorics), Tur\'an generalized Mantel's result to independent
sets of arbitrary size. He also asked if similar generalizations can
be obtained for hypergraphs, and these questions became notoriously
known ever since as one of the most difficult open problems in
discrete mathematics.

To be more specific, for a family $\scr H$ of $r$-uniform hypergraphs ({\em $r$-graphs} in what follows) and an integer $n$, let $\text{ex}_{\min}(n; \scr H)$ be the minimal possible number of edges  in an $n$-vertex $r$-graph not containing any of $H\in\scr H$ as an induced subgraph, and let

$$
\pi_{\min}(\scr H) \df \lim_{n\rightarrow\infty}
\frac{\text{ex}_{\min}(n; \scr H)}{{n\choose r}}
$$
(it is well-known that this limit exists). For $\ell>r\geq 2$, let $I^r_\ell$ be the empty $r$-graph on $\ell$ vertices. Then we still do not know $\pi_{\min}(I_\ell^r)$ for {\em any} pair $\ell>r\geq 3$. More information on the history and state of the art for this and related problems can be found in the recent comprehensive survey \cite{Kee}, and we will henceforth concentrate on the  simplest case  $r=3,\ \ell=4$. Tur\'an's conjecture says that $\pi_{\min}(I^3_4)=4/9$, and it is also sometimes called {\em Tur\'an's $(3,4)$-problem} or {\em tetrahedron problem}. De Caen \cite{Dec}, Giraud (unpublished) and Chung and Lu \cite{GrL} proved
increasingly stronger lower bounds on $\pi_{\min}(I^3_4)$, with the
current record being
$$
\pi_{\min}(I^3_4) \geq 0.438334
$$
\cite{turan, FaV}.

A prominent way to attack a difficult extremal problem is by first better understanding the nature of its (conjectured) extremal configurations and then gradually trying to solve this problem in a ``neighborhood'' of this set which is as ``large'' and ``natural'' as possible (we deliberately leave many terms in this sentence imprecise). In the context of Turan's (3,4)-problem, the first part of this program can be reasonably said to have been successfully completed in the series of classical papers by Tur\'an himself \cite{Tur}, Brown \cite{Bro}, Kostochka \cite{Kos} and  Fon-der-Flaass \cite{Fon}.

For the second part, Razborov \cite{turan} proved that
\begin{equation}\label{drawback}
    \pi_{\min}(I^3_4,G_3)= 4/9,
\end{equation}
where $G_3$ is the 3-graph on 4 vertices with 3 edges. In another paper \cite{fdf} of the same author, Tur\'an (3,4)-problem was settled for two broad classes of 3-graphs resulted from the Fon-der-Flaas  construction (its details will be reviewed below), each of these two classes including {\em all Tur\'an-Brown-Kostochka examples}.

On the contrary, the result \eqref{drawback} has the obvious limitation that the additional 3-graph $G_3$ is missing only in Tur\'an's original example. In fact, Pikhurko \cite{Pik} proved that his example is essentially the {\em only} extremal configuration for the extremal problem \eqref{drawback}.

\bigskip
The main purpose of this note is to circumvent this by showing that
\begin{equation}\label{new}
    \pi_{\min}(I^3_4,H_1,H_2,H_3)= 4/9,
\end{equation}
where $H_i$ are 3-graphs on 5 vertices with the following sets of 3-edges:
\begin{equation} \label{forbidden}
\longeq{
E(H_1) &\df& \{(123) (124) (134) (234) (125) (345)\}\\
E(H_2) &\df& \{(123) (124) (134) (234) (135) (145) (235) (245)\}\\
E(H_3) &\df& \{(123) (124) (134) (234) (125)(135) (145) (235) (245)\}}
\end{equation}
In words: in each of these 3-graphs, $\{1,2,3,4\}$ span a clique, and the link of the remaining vertex 5 is either a perfect matching ($H_1$) or the complement to a perfect matching ($H_2$) or the complement to a single edge ($H_3$).
Unlike $G_3$, our new forbidden 3-graphs {\em are} missing in all Tur\'an-Brown-Kostochka configurations (see Claim \ref{missing} below).

While we do not have immediate ideas how to get rid of $H_1,H_2,H_3$ in \eqref{new} (obvious attempts encounter the same kind of complications that in general make extremal problems for hypergraphs so painstakingly difficult), we would like to emphasize that our methods use some (apparently) novel ideas that might be of independent interest. Specifically, we are trying to utilize one of the results in \cite{fdf} that solves Turan's (3,4)-problem for the class of 3-graphs resulting, via  Fon-der-Flaas  interpretation, from a class of oriented graphs. In the ideal world, one would have to take an arbitrary $(I^3_4,H_1,H_2,H_3)$-free 3-graph $G$ and somehow create an useful oriented graph from this class to which we could apply that result. We, however, doubt very much that this goal can be easily achieved in its entirety.

Our (apparently) new observation is that our final purpose is so limited that we actually need something much more modest and relaxed than this ``global'' structure. Namely, we can assume that there exists a {\em fixed} 3-graph $H$ that is {\em not} realizable via the Fon-der-Flaass construction and that has a {\em constant density} in a hypothetical counterexample $G$. The latter fact allows us to apply standard machinery from Ramsey Theory to find in $G$ several extremely well-positioned copies of $H$. And then we attempt to create the Fon-der-Flaass structure on $H$ by retrieving information encoded  in this gadget. It turns out that $H_1,H_2,H_3$ are precisely the only obstacles for completing this project (that, as we assumed from the beginning, can not be completed).

As we said, this method appears to be rather general and uses very little of the specifics of the Turan's (3,4)-problem. This, however, does not  apply to the last (information retrieval) part that, on the contrary, seems to shed at least some light on the otherwise quite mysterious nature of the Kostochka-Fon-der-Flaass examples.

\medskip
We also include two more results proved via routine calculations in Flag Algebras. The first oimproves upon \eqref{drawback} by replacing $G_3$ with the 3-graph having the edge set $\{(123)(124) (134) (235)\}$, and the second establishes the better bound $\pi_{\min}(I^3_4, M_2)\geq 0.4557$, where $M_2$ is the only $I^3_4$-free 3-graph with six vertices and six edges.

\section{Preliminaries}

Let $[n]\df \{1,2,\ldots,n\}$, and, for a finite set $X$, let $[X]^k$ be the collection of its $k$-element subsets. $[[n]]^k$ is abbreviated to $[n]^k$.

In this paper we will be working with 3-graphs that will be normally denoted by the letters $G,H$, possibly with indices, and with oriented graphs, or {\em orgraphs} (directed graphs without loops, multiple or anti-parallel edges) typically denoted by $\Gamma$. In both cases, $V(\cdot)$ will be the set of vertices, and $E(\cdot)$ will be the set of 3-edges/oriented edges. For $V_0\subseteq V(G)$, $G|_{V_0}$ is the 3-graph induced on $V_0$, and likewise for $\Gamma|_{V_0}$. Two vertices $u,v$ are {\em independent} in an orgraph $\Gamma$ if neither $\edge uv$ nor $\edge vu$ is in $E(\Gamma)$.

Given a finite set $\scr H$ of $3$-graphs and an integer $n$, let $\text{ex}_{\min}(n; \scr H)$ be the minimal possible number of edges  in an $n$-vertex $r$-graph not containing any of $H\in\scr H$ as an induced subgraph, and let

$$
\pi_{\min}(\scr H) \df \lim_{n\rightarrow\infty}
\frac{\text{ex}_{\min}(n; \scr H)}{{n\choose r}}
$$
(it is well-known that this limit exists). Let $I^3_4$ be the empty $3$-graph on four vertices, and let $H_1,H_2,H_3$ be the 3-graphs defined by their edge sets in \eqref{forbidden}.

Then our main result reads as follows:

\begin{theorem} \label{main}
$\pi_{\min}(I^3_4, H_1, H_2, H_3)=4/9$.
\end{theorem}

Since, besides developing some interesting techniques, our primary motivation stems from the fact that $H_1,H_2,H_3$ are missing in all Kostochka configurations, we begin with a (relatively easy) clarification and verification of this fact. The reader interested only in the proof of Theorem \ref{main} may safely skip this digression.

Let $\Omega={\Bbb Z_3}\times {\Bbb R}$, and consider the (infinite) orgraph $\Gamma_K=(\Omega,E_K)$ given by
$$
E_K\df \set{\langle (a,x),(b,y)\rangle}{(x+y<0 \land b=a+1) \lor (x+y>0 \land b=a-1)}.
$$
$\Gamma_K$ does not contain induced oriented cycles $\vec C_4$.

For an arbitrary $\vec C_4$-free orgraph $\Gamma$, let  $FDF(\Gamma)$ be the 3-graph with $V(FDF(\Gamma))=V(\Gamma)$ in which $(uvw)$ is declared to span an edge if and only if $\Gamma|_{\{u,v,w\}}$ either contains an isolated (that is, of both in-degree and out-degree 0) vertex, or contains a vertex of out-degree 2 ({\em Fon-der-Flaass interpretation} \cite{Fon}). Then $FDF(\Gamma)$ is $I^3_4$-free, and all known extremal configurations when the number of vertices is divisible by three are precisely of the form $FDF(\Gamma_K)|_{{\Bbb Z}_3\times S}$ for a finite $S\subseteq {\Bbb R}$.

\begin{claim} \label{missing}
$FDF(\Gamma_K)$ does not contain induced copies of $H_1, H_2, H_3$.
\end{claim}
\begin{proof}
Assume the contrary, and let $\omega_i = (a_i,x_i)\ (i\in [5])$ defines an induced embedding of some $H\in \{H_1,H_2,H_3\}$ into $FDF(\Gamma_K)$. We assume w.l.o.g. that the real numbers $x_i$ are pairwise distinct. We treat the cases $H\in \{H_1,H_2\}$ and $H=H_3$ separately.

\smallskip\noindent
{\bf Case 1.} {\sc $H=H_1$ or $H=H_2$}.

$\{\omega_1,\omega_2,\omega_3,\omega_4\}$ is a clique in $FDF(\Gamma_K)$. It is easy to see then that in the orgraph $\Gamma_K$ these four vertices must span one of the three orgraphs on Figure \ref{cliques}.

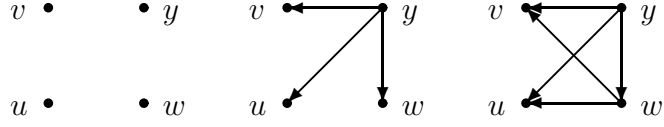
\begin{figure}[tb]
\begin{center}
\setlength{\unitlength}{0.254mm}
\begin{picture}(353,66)(20,-73)
        \special{color rgb 0 0 0}\allinethickness{0.254mm}\special{sh 0.99}\put(40,-15){\ellipse{4}{4}} 
        \special{color rgb 0 0 0}\allinethickness{0.254mm}\special{sh 0.99}\put(40,-65){\ellipse{4}{4}} 
        \special{color rgb 0 0 0}\allinethickness{0.254mm}\special{sh 0.99}\put(90,-15){\ellipse{4}{4}} 
        \special{color rgb 0 0 0}\allinethickness{0.254mm}\special{sh 0.99}\put(90,-65){\ellipse{4}{4}} 
        \special{color rgb 0 0 0}\put(20,-21){\shortstack{$v$}} 
        \special{color rgb 0 0 0}\put(20,-71){\shortstack{$u$}} 
        \special{color rgb 0 0 0}\put(100,-71){\shortstack{$w$}} 
        \special{color rgb 0 0 0}\put(100,-21){\shortstack{$y$}} 
        \special{color rgb 0 0 0}\allinethickness{0.254mm}\special{sh 0.99}\put(165,-15){\ellipse{4}{4}} 
        \special{color rgb 0 0 0}\allinethickness{0.254mm}\special{sh 0.99}\put(165,-65){\ellipse{4}{4}} 
        \special{color rgb 0 0 0}\allinethickness{0.254mm}\special{sh 0.99}\put(215,-15){\ellipse{4}{4}} 
        \special{color rgb 0 0 0}\allinethickness{0.254mm}\special{sh 0.99}\put(215,-65){\ellipse{4}{4}} 
        \special{color rgb 0 0 0}\put(145,-21){\shortstack{$v$}} 
        \special{color rgb 0 0 0}\put(145,-71){\shortstack{$u$}} 
        \special{color rgb 0 0 0}\put(225,-71){\shortstack{$w$}} 
        \special{color rgb 0 0 0}\put(225,-21){\shortstack{$y$}} 
        \special{color rgb 0 0 0}\allinethickness{0.254mm}\special{sh 0.99}\put(290,-15){\ellipse{4}{4}} 
        \special{color rgb 0 0 0}\allinethickness{0.254mm}\special{sh 0.99}\put(290,-65){\ellipse{4}{4}} 
        \special{color rgb 0 0 0}\allinethickness{0.254mm}\special{sh 0.99}\put(340,-15){\ellipse{4}{4}} 
        \special{color rgb 0 0 0}\allinethickness{0.254mm}\special{sh 0.99}\put(340,-65){\ellipse{4}{4}} 
        \special{color rgb 0 0 0}\put(270,-21){\shortstack{$v$}} 
        \special{color rgb 0 0 0}\put(270,-71){\shortstack{$u$}} 
        \special{color rgb 0 0 0}\put(350,-71){\shortstack{$w$}} 
        \special{color rgb 0 0 0}\put(350,-21){\shortstack{$y$}} 
        \special{color rgb 0 0 0}\allinethickness{0.254mm}\put(215,-15){\vector(-1,0){50}} 
        \special{color rgb 0 0 0}\allinethickness{0.254mm}\put(215,-15){\vector(0,-1){50}} 
        \special{color rgb 0 0 0}\allinethickness{0.254mm}\put(215,-15){\vector(-1,-1){50}} 
        \special{color rgb 0 0 0}\allinethickness{0.254mm}\put(340,-15){\vector(-1,0){50}} 
        \special{color rgb 0 0 0}\allinethickness{0.254mm}\put(340,-15){\vector(-1,-1){50}} 
        \special{color rgb 0 0 0}\allinethickness{0.254mm}\put(340,-15){\vector(0,-1){50}} 
        \special{color rgb 0 0 0}\allinethickness{0.254mm}\put(340,-65){\vector(-1,1){50}} 
        \special{color rgb 0 0 0}\allinethickness{0.254mm}\put(340,-65){\vector(-1,0){50}} 
        \special{color rgb 0 0 0} 
\end{picture}
\caption{\label{cliques} $\{u,v,w,y\}=\{\omega_1,\omega_2,\omega_3,\omega_4\}$}
\end{center}
\end{figure}

If $\{\omega_1,\omega_2,\omega_3,\omega_4\}$ are independent in $\Gamma_K$ (that is, if $a_1=a_2=a_3=a_4$), then $\omega_5$ may not be independent of them, and its {\em link} $L$ in $FDF(\Gamma_K)|_{\{\omega_1,\omega_2,\omega_3,\omega_4\}}$ defined as
$$
L\df\set{e\in [\{\omega_1,\omega_2,\omega_3,\omega_4\}]^2}{e \cup \{\omega_5\} \in E(FDF(\Gamma_K))}
$$
will be a clique on the set
\begin{equation} \label{clique}
\set{z\in \{\omega_1,\omega_2,\omega_3,\omega_4\}}{\edge{\omega_5}z\in E(\Gamma_K)}.
\end{equation}
None of the graphs $H_1, H_2$ has this form.

Assume now that $\{\omega_1,\omega_2,\omega_3,\omega_4\}$ span $\vec K_{1,3}$ in $\Gamma_K$, as shown on Figure \ref{cliques}, middle picture. If $\omega_5$ is independent from $\{u,v,w\}$ then $\{u,v,w,\omega_5\}$ is another clique in $FDF(\Gamma_K)$, if $\omega_5$ is independent from $y$ then $y$ is an isolated vertex in the link $L$, and if $\edge y{\omega_5}\in E(\Gamma_K)$ then $y$ has degree 3 in $L$. None of these may happen in $H_1$ or $H_2$ so we are left with the case when $\edge{\omega_5}y\in E(\Gamma_K)$ and $\omega_5$ is also connected to the vertices $u,v,w$. The again $L$ will be the clique on the same set \eqref{clique} which is impossible. This completes the analysis of the second case on Figure \ref{cliques}.

Finally, assume that $\{\omega_1,\omega_2,\omega_3,\omega_4\}$ span in $\Gamma_K$ the third orgraph from Figure \ref{cliques}. By the same token as in the previous case, $\omega_5$ may not be independent of $y$ and $\edge y{\omega_5}\not\in E(\Gamma_K)$. Thus, as before, we necessarily have $\edge{\omega_5}y\in E(\Gamma_K)$.

If $\omega_5$ is independent from $u,v$, then $\edge{\omega_5}w\in E(\Gamma_K)$ since otherwise $y$ again would be isolated in the link $L$ of $\omega_5$. But then the configuration spanned by $y,w,u,\omega_5$ is impossible in $\Gamma_K$. Indeed, if, say, $u=\omega_i$ then the two vertices $y,w$ force the opposite inequalities between $x_i$ and $x_5$ (recall that $\omega_j=(a_j,x_j)\in {\Bbb Z}_3\times {\Bbb R}$).

Finally, if $\omega_5$ is independent of $w$, then either $\edge{\omega_5}u\in E(\Gamma_K)$ or $\edge{\omega_5}v\in E(\Gamma_K)$ (otherwise $y$ would be isolated in $L$), and we assume w.l.o.g. that $\edge{\omega_5}u\in E(\Gamma_K)$. But now $\edge{\omega_5}v\not\in E(\Gamma_K)$ (as otherwise $(v,u,y)$ would be a triangle in $L$) and $\edge v{\omega_5}\not\in E(\Gamma_k)$ (as otherwise $(v,w,y)$ would be an independent set in $L$). This contradiction completes the analysis in Case 1.

\smallskip\noindent
{\bf Case 2.} {\sc $H=H_3$}.

Since $(\omega_3\omega_4\omega_5)\not \in E(FDF(\Gamma_K))$, either these three vertices span in $\Gamma_K$ the oriented cycle $\vec C_3$ or, after a suitable re-numeration,  $\edge{\omega_3}{\omega_4}\in E(\Gamma_K)$ while $\omega_3$ and $\omega_5$ are independent. In the first case, if $\omega_1$ is independent of (say) $\omega_3$ and (say) $\edge{\omega_3}{\omega_4}\in E(\Gamma_K)$, this implies one more missing edge $(\omega_1\omega_3\omega_4)$ in $FDF(\Gamma_K)$. In the second case, by the same token neither of $\omega_1,\omega_2$ can be independent of $\omega_3,\omega_5$. But then at least two of the vertices $\omega_1,\omega_2,\omega_4$ must be independent, and then these two vertices along with $\omega_3,\omega_5$ may span at most two edges in $H$ which may not happen since $H=H_3$. Claim \ref{missing} is proved.
\end{proof}

Along with $\vec C_4$, the orgraph $\Gamma_K$ does not contain induced copies of $\bar P_3$ (an oriented edge plus an isolated vertex). One of the two main results in \cite{fdf} settles Turan's (3,4)-problem for 3-graphs obtained via the Fon-der-Flaass interpretation from an oriented graph with the latter property. Before stating this result here, let us introduce a (tiny) bit of the flag algebra formalism \cite{flag} that we will need in the next section anyway.

For two 3-graphs $H,G$ with $|V(H)|\leq |V(G)|$ we let $p(H,G)$ denote the {\em density} of induced copies of $H$ in $G$. We let $\rho$ denote the 3-graph on three vertices consisting of a single edge; thus, $p(\rho, G)$ is simply the edge density of $G$. We have the following basic {\em chain rule} (see e.g. \cite[Lemma 2.2]{flag}):
\begin{equation} \label{chain_rule}
p(H, \widehat G)= \sum_{G\in\scr M_\ell}p(H,G)p(G,\widehat G),
\end{equation}
where $|V(H)|\leq\ell\leq |V(\widehat G)|$ and $\scr M_\ell$ is the set of all 3-graphs on $\ell$ vertices, up to an isomorphism.

\begin{proposition}[\protect{\cite[Theorem 2.3]{fdf}}] \label{fdf}
For any increasing sequence $\{\Gamma_n\}$ of oriented graphs without induced copies of $\vec C_4$ or $\bar P_3$, $p(\rho, FDF(\Gamma_n))\geq \frac 49(1-0(1))$.
\end{proposition}

Next, we need a multi-partite version of Ramsey's theorem that we formulate here at the level of generality sufficient for our purposes.

\begin{proposition}[\protect{\cite[Theorem 5.1.5]{GrRS}}] \label{ramsey}
For any $\ell>0$, $n>0$ and $r_1,\ldots,r_\ell>0$ there exists $N>0$ such that if $|B_i|=N\ (1\leq i\leq\ell)$ and $[B_1]^{r_1}\times\cdots\times [B_\ell]^{r_\ell}$ is colored in two colors then there exist $A_i\subseteq B_i\ (|A_i|=n)$ such that $[A_1]^{r_1}\times\cdots\times [A_\ell]^{r_\ell}$ is monochromatic.
\end{proposition}

It is easy to iterate this statement to get the following.

\begin{claim} \label{turan_ramsey}
For any $\ell,n,r>0$ there exists $N>0$ such that the following holds. Let $B= B_1\stackrel .\cup\ldots\stackrel .\cup B_\ell$, where $|B_i|=N$, and assume that $[B]^r$ is colored in two colors. Then there exist $A_i\subseteq B_i\ (|A_i|=n)$ such that for any $E\in [A_1\cup\ldots\cup A_\ell]^r$, its color depends only on the tuple of cardinalities $\langle |E\cap A_1|,\ldots, |E\cap A_\ell|\rangle$.
\end{claim}
\begin{proof}
For every partition $r=r_1+\cdots+r_\ell\ (r_i\geq 0)$, our original coloring of $[B]^r$ in a natural way induces a coloring of $[B_1]^{r_1}\times\cdots\times [B_\ell]^{r_\ell}$. Now we simply apply recursively Proposition \ref{ramsey} to all these partitions (arbitrarily ordered).
\end{proof}

When $r_1=\ldots=r_\ell=1$, Proposition \ref{ramsey} also has a density version that, in a slightly different form, has been extensively used in areas like Additive Combinatorics, Extremal Combinatorics and Complexity Theory.

\begin{proposition}[\protect{\cite[Theorem 5.1.4]{GrRS}}] \label{detupling}
For all $\ell>0$, $n>0$ and $\delta>0$ there exists $N_0>0$ so that if $|B_i|=N\ (1\leq i\leq\ell)$ with $N\geq N_0$ and $S\subseteq B_1\times\cdots\times B_\ell$ has $|S|\geq\delta N^\ell$, then there exist $A_i\subseteq B_i\ (|A_i|=n)$ such that $A_1\times\cdots\times A_\ell \subseteq S$.
\end{proposition}

Now we have all the ingredients necessary to prove our main result.

\section{Proof of Theorem \ref{main}}

Fix an increasing sequence $\{G_m\}$ of 3-graphs not containing $I^3_4,H_1, H_2, H_3$ as induced subgraphs. We have to prove that $\liminf_{m\to\infty} p(\rho,G_m)\geq 4/9.$ Assume the contrary, then (by restricting to a sub-sequence) we can assume w.l.o.g. that
\begin{equation} \label{upper}
p(\rho, G_m)\leq 4/9-\epsilon
\end{equation}
for a fixed $\epsilon >0$ and all $m$.

Let us call a 3-graph $G$ {\em regular} if $E(G)\supseteq E(FDF(\Gamma))$ for some orgraph $\Gamma$ on the same set of vertices $V(G)$ without induced copies of $\vec C_4$ or $\bar P_3$, and {\em singular} otherwise. By Proposition \ref{fdf}, there exists an integer $\ell$ such that for {\em every} regular 3-graph $G$ on $\ell$ vertices,
\begin{equation} \label{lower}
p(\rho, G)\geq \frac 49-\frac{\epsilon}2.
\end{equation}

Fix for a moment an integer $m$ such that $|V(G_m)|\geq\ell$ and let
$$
R\df\sum_{{G\in\scr M_\ell}\atop {G\ \text{is regular}}} p(G,G_m)
$$
be the probability that a randomly chosen $\ell$-vertex induced subgraph of $G_m$ is regular. Given the formula \eqref{chain_rule} with $H:=\rho$ and $\widehat G:=G_m$, we get from \eqref{upper} and \eqref{lower} that $(4/9-\epsilon)\geq R\of{\frac 49-\frac{\epsilon}2}$ which implies that $1-R\geq\epsilon$. Hence, for some absolute (not depending on $m$) constant $\delta>0$, there exists a {\em singular} 3-graph $G\in\scr M_\ell$ such that $p(G, G_m)\geq\delta$.

Now we let $m$ vary. Since $\scr M_\ell$ is finite, we can assume w.l.o.g. that this singular 3-graph $G$ is the same for all $m$. And now we are going to apply the ``regularization'' machinery reviewed at the end of the previous section to arrive at a contradiction with singularity of $G$.

In Claim \ref{turan_ramsey} we set $n:=2$, $r:=3$, and let $N_1$ be the resulting bound. Next, we set
\begin{equation} \label{delta_prime}
\delta'=\frac 12 \ell^{-\ell}\delta,
\end{equation}
and apply Proposition \ref{detupling} with $n:=N_1$ and $\delta:=\delta'$. Let $N_0$ be the resulting bound, and now we fix $m$ such that $|V(G_m)|\geq \ell N_0$. W.l.o.g. we may assume that $|V(G_m)|$ is divisible by $\ell$, and let $N\df \frac 1\ell |V(G_m)|$. Note for the record that $N\geq N_0$.

Let $V(G)=[\ell]$. Consider a random balanced partition $V(G_m)=\rn{B_1}\stackrel .\cup \ldots\stackrel .\cup \rn{B_\ell}$ into $N$-sets. By a standard averaging argument, the expectation of the density of induced embeddings $\alpha\function G{G_m}$ such that $\alpha(i)\in \rn{B_i}\ (i\in [\ell])$ is at least $\delta'$ (recall that $\delta'$ is given by \eqref{delta_prime}). Fix an arbitrary balanced partition $V(G_m)=B_1\stackrel .\cup \ldots\stackrel .\cup B_\ell$ with this property, and let $S\subseteq [B_1]\times\ldots\times [B_\ell]$ consist of those tuples $(v_1,\ldots,v_\ell)$ for which the mapping  $\alpha\function{[\ell]}{V(G_m)}$ given by $\alpha(i) = v_i$ does define an induced embedding of $G$.

Applying Proposition \ref{detupling}, we find $A_i\subseteq B_i$ with $|A_i|=N_1$ and $A_1\times\ldots\times A_\ell\subseteq S$. And applying Claim \ref{turan_ramsey} (with $B_i:=A_i$) to the 2-coloring of $[A_1\cup\ldots\cup A_\ell]^3$ defined by the set of edges of $G_m$, we find distinct pairs of elements $a_i,b_i\in A_i$ such that for any $i\neq j\in [\ell]$,
$$
(a_ib_ia_j) \in E(G_m) \equiv (a_ib_ib_j) \in E(G_m).
$$
Note also for the record that for any $1\leq i<j<k\leq\ell$ and for any $c_i\in \{a_i, b_i\}$, $c_j\in \{a_j, b_j\}$,  $c_k\in \{a_k, b_k\}$, $(c_ic_jc_k)\in E(G_m)$ if and only if $(ijk)\in E(G)$, for any choice of the representatives $c_i, c_j, c_k$.

We define the directed graph $\Gamma$ on $[\ell]$ by introducing a directed edge $\edge ij$ if and only if $(a_ib_ia_j), (a_ib_ib_j)\not \in E(G_m)$ (note the negation!) As the first observation, since $\{a_i,b_i,a_j, b_j\}$ span at least one 3-edge, $\edge ij$ and $\edge ji$ can not simultaneously belong to $E(\Gamma)$, therefore $\Gamma$ is actually an oriented graph.

Now we claim that $E(G)\supseteq E(FDF(\Gamma))$ and that $\Gamma$ does not contain induced copies of $\bar P_3$ and $\vec C_4$; this will contradict the singularity of $G$. In the case analysis below (that we split into a sequence of simple claims), $i,j,k$ stand for arbitrary pairwise different elements of $[\ell]$.

\begin{claim}
If $\edge ij\in E(\Gamma)$ and $\edge ik\in E(\Gamma)$ then $(ijk)\in E(G)$.
\end{claim}
\begin{proof}
Since $\{a_i, b_i, a_j, a_k\}$ is not independent in $G_m$, at least one of the two triples $(a_ia_ja_k)$, $(b_ib_jb_k)$ must be in $E(G_m)$ and this implies $(ijk)\in E(G)$.
\end{proof}

The remaining analysis does require the assumption that the 3-graphs $H_1,H_2,H_3$ are missing.

\begin{claim} \label{two}
If $i,j$ are independent in $\Gamma$, and $i,k$ are also independent in $\Gamma$, then $(ijk)\in E(G)$.
\end{claim}
\begin{proof}
Note first that the assumption of our claim simply says that both sets $\{a_i, b_i, a_j, b_j\}$ and $\{a_i, b_i, a_k, b_k\}$ span a clique in $G_m$. By symmetry, we can assume w.l.o.g. that $\edge jk\not\in E(\Gamma)$, that is, $(a_jb_ja_k), (a_jb_jb_k)\in E(G_m)$. But then since $\{a_i, b_i, a_j, b_j, a_k\}$ does not span a copy of $H_1$, at least one of the four edges $(c_ic_ja_k)\ (c_i\in \{a_i, b_i\},\ c_j\in\{a_j,b_j\})$ must be present in $G_m$ which implies $(ijk)\in E(G)$.
\end{proof}

Reviewing the definition of $FDF(\Gamma)$, we see that we have proved $E(FDF(\Gamma))\subseteq E(G)$. We still have to verify that $\Gamma$ is $\bar P_3$-free and $\vec C_4$-free.

\begin{claim} \label{three}
If $i,j$ are independent in $\Gamma$ and $\langle j,k\rangle\in E(\Gamma)$ then $(ijk)\not\in E(G)$.
\end{claim}
\begin{proof}
We again look at the configuration spanned by $\{a_i,b_i,a_j,b_j,a_k\}$. Assuming $(ijk)\in E(G)$, the only 3-edges that are missing here are $(a_jb_ja_k)$ and, possibly, $(a_ib_ia_k)$. Which gives us either $H_2$ or $H_3$, and this contradiction proves that $(ijk)\not\in E(G)$.
\end{proof}

Now, $\Gamma$ may not contain an induced copy of $\bar P_3$: the case when the three vertices spanning $\bar P_3$ in $\Gamma$ do not form an edge of $G$ is taken care of by Claim \ref{two}, and the case when it is an edge is ruled out by Claim \ref{three}. Also, $\Gamma$ may not contain a copy of $\vec C_4$: since these four vertices may not be independent in $G$, three of them must form an edge which is again in contradiction with Claim \ref{three}.

We have arrived at the desired contradiction by constructing a Fon-der-Flaass realization $\Gamma$ for a spanning subgraph of the 3-graph $G$. This completes the proof of our main result.

\section{Miscellaneous calculations}

We include here two other results of a similar flavor, albeit seemingly less interesting: the additional forbidden 3-graphs are present in $FDF(\Gamma_K)$, and, in the second case, are present even in Turan's original example.

Our first statement improves upon \eqref{drawback}.

\begin{theorem} \label{drawback_improved}
$\pi_{\min}(I^3_4, H_4)=4/9$, where $H_4$ is a 3-graph on $[5]$ with $E(H_4)=\{(123)(124) (234) (135)\}$.
\end{theorem}

Another viable strategy in approaching difficult extremal problems (see e.g. \cite{FPS}) is to try to find in a hypothetical counterexample increasingly large pieces that ``correspond'' to known extremal configurations. For Turan's (3,4)-problem we have been able to verify the first step on this road.
\begin{theorem}
$\pi_{\min}(I^3_4, M_2)\geq 0.4557$, where $M_2\df FDF(\Gamma_K)|_{{\Bbb Z}_3\times \{1,2\}}$ is the only extremal configuration on 6 vertices.
\end{theorem}

Since these days an interested reader can check statements like this using the publicly available Flagmatic software \cite{FaV}, we do not present here  (tedious!) results of our calculations.

\bibliographystyle{alpha}
\bibliography{razb}

\begin{thebibliography}{FRV12}

\bibitem[Bro83]{Bro}
W.~G. Brown.
\newblock On an open problem of {Paul Tur\'an} concerning 3-graphs.
\newblock In {\em Studies in pure mathematics}, pages 91--93. {Birkh\"auser},
  1983.

\bibitem[Cae91]{Dec}
D.~de Caen.
\newblock The current status of {Tur\'an} problem on hypergraphs.
\newblock In {\em Extremal Problems for Finite Sets, Visegr\'ad (Hungary)},
  volume~3, pages 187--197. Bolyai Society Mathematical Studies, 1991.

\bibitem[CL99]{GrL}
F.~Chung and L.~Lu.
\newblock An upper bound for the {Tur\'an} number $t_3(n, 4)$.
\newblock {\em Journal of Combinatorial Theory (A)}, 87:381--389, 1999.

\bibitem[FdF88]{Fon}
D.~G. Fon-der Flaass.
\newblock Method for construction of (3,4)-graphs.
\newblock {\em Mathematical Notes}, 44(4):781--783, 1988.
\newblock Translated from {\em Matematicheskie Zametki}, Vol. 44, No. 4, pp.
  546-550, 1988.

\bibitem[FPS03]{FPS}
Z.~{F\"uredi}, O.~Pikhurko, and M.~Simonovits.
\newblock The {Tur\'an} density of the hypergraph $\{abc,ade,bde.cde\}$.
\newblock {\em Electronic Journal of Combinatorics}, R18, 2003.

\bibitem[FRV12]{FaV}
V.~Falgas-Ravry and E.~R. Vaughan.
\newblock On applications of {Razborov}'s flag algebra calculus to extremal
  3-graph theory.
\newblock Technical Report 1110.1623v2 [math.CO], arXiv, 2012.

\bibitem[GRS90]{GrRS}
R.~L. Graham, B.~L. Rothschild, and J.~H. Spencer.
\newblock {\em {Ramsey} theory, 2nd edition}.
\newblock Wiley-Interscience, 1990.

\bibitem[Kee11]{Kee}
Hypergraph {Tur\'an} problems.
\newblock In R.~Chapman, editor, {\em Surveys in Combinatorics}, pages 83--140.
  Cambridge University Press, 2011.

\bibitem[Kos82]{Kos}
A.~V. Kostochka.
\newblock A class of constructions for {Tur\'an's} (3, 4)-problem.
\newblock {\em Combinatorica}, 2(2):187--192, 1982.

\bibitem[Man07]{Man}
W.~Mantel.
\newblock Problem 28, solution by {H. Gouwentak, W. Mantel, J. Teixeira de
  Mattes, F. Schuh and W.A. Wythoff}.
\newblock {\em Wiskundige Opgaven}, 10:60--61, 1907.

\bibitem[Pik11]{Pik}
O.~Pikhurko.
\newblock The minimum size of 3-graphs without a 4-set spanning no or exactly
  three edges.
\newblock {\em European Journal of Combinatorics}, 23, 2011.

\bibitem[Raz07]{flag}
A.~Razborov.
\newblock Flag algebras.
\newblock {\em Journal of Symbolic Logic}, 72(4):1239--1282, 2007.

\bibitem[Raz10]{turan}
A.~Razborov.
\newblock On 3-hypergraphs with forbidden 4-vertex configurations.
\newblock {\em SIAM Journal on Discrete Mathematics}, 24(3):946--963, 2010.

\bibitem[Raz11]{fdf}
A.~Razborov.
\newblock On the {Fon-der-Flaass} interpretation of extremal examples for
  {Turan}'s (3,4)-problem.
\newblock {\em Proceedings of the Steklov Institute of Mathematics},
  274:247--266, 2011.

\bibitem[Tur41]{Tur}
P.~Tur\'an.
\newblock {Egy gr\'afelm\'eleti sz\'els\"o\'ert\'ekfeladatr\'ol}.
\newblock {\em Mat. \'es Fiz. Lapok}, 48:436--453, 1941.

\end{thebibliography}
\end{document}